\documentclass[12pt]{amsart}

\usepackage{amssymb,enumerate,amsthm,yfonts,mathrsfs}
\usepackage{amsmath}
\usepackage{bbm}           % used for blackboard 1
\usepackage{bm} 
\usepackage[colorlinks=true, pdfstartview=FitV, linkcolor=blue, citecolor=blue, urlcolor=blue]{hyperref}

\makeatletter
\def\Ddots{\mathinner{\mkern1mu\raise\p@
\vbox{\kern7\p@\hbox{.}}\mkern2mu
\raise4\p@\hbox{.}\mkern2mu\raise7\p@\hbox{.}\mkern1mu}}
\makeatother

\def\Xint#1{\mathchoice
{\XXint\displaystyle\textstyle{#1}}%
{\XXint\textstyle\scriptstyle{#1}}%
{\XXint\scriptstyle\scriptscriptstyle{#1}}%
{\XXint\scriptscriptstyle\scriptscriptstyle{#1}}%
\!\int}
\def\XXint#1#2#3{{\setbox0=\hbox{$#1{#2#3}{\int}$}
\vcenter{\hbox{$#2#3$}}\kern-.5\wd0}}

\def\dashint{\Xint-}

\def\Xint#1{\mathchoice
   {\XXint\displaystyle\textstyle{#1}}%
   {\XXint\textstyle\scriptstyle{#1}}%
   {\XXint\scriptstyle\scriptscriptstyle{#1}}%
   {\XXint\scriptscriptstyle\scriptscriptstyle{#1}}%
   \!\int}
\def\XXint#1#2#3{{\setbox0=\hbox{$#1{#2#3}{\int}$}
     \vcenter{\hbox{$#2#3$}}\kern-.5\wd0}}

\def\avgint{\Xint-}

\newtheorem{theorem}{Theorem}[section]

\newtheorem{conjecture}[theorem]{Conjecture}
\newtheorem{corollary}[theorem]{Corollary}

\theoremstyle{definition}

\newtheorem{lemma}[theorem]{Lemma}

\newtheorem{prop}[theorem]{Proposition}
\newtheorem{remark}[theorem]{Remark}

\def\al{{\alpha}}
\def\R{\mathbb R}

\def\Z{\mathbb Z}

\def\ra{\rightarrow}
\def\bey{\begin{eqnarray*}}
\def\eey{\end{eqnarray*}}
 
 \allowdisplaybreaks

\def\D{{\mathscr D}}
\def\Ps{{\mathscr P}}

\def\Q{{\mathcal Q}}
\def\Ca{{\mathcal C}}

\def\MM{{\mathscr M}}

\def\Sp{{\mathcal S}}

\begin{document}

\title{One and two weight norm inequalities for Riesz potentials}
\author{David Cruz-Uribe, SFO and Kabe Moen}

%}

 \address{David Cruz-Uribe, SFO, Department of Mathematics, Trinity College, Hartford, CT 06106, USA}
\email{david.cruzuribe@trincoll.edu}

 \address{Kabe Moen, Department of Mathematics, University of Alabama, Tuscaloosa, AL 35487-0350}
\email{kabe.moen@ua.edu}

\thanks{The first author is supported by the Stewart-Dorwart faculty
  development fund at Trinity College and by
  grant MTM2012-30748 from the Spanish Ministry of Science and
  Innovation.  The second author is supported by NSF Grant 1201504}



\date{November 8, 2012}

%\setpagewiselinenumbers
%\linenumbers
\maketitle

\begin{abstract}
We consider weighted norm inequalities for the Riesz potentials
$I_\alpha$, also referred to as fractional integral operators.  First we prove mixed $A_p$-$A_\infty$ type estimates in the spirit of~\cite{hytonen-lacey-IUMJ,hytonen-perez-analPDE,lacey-HJM}.  Then we prove strong and weak type inequalities in the case
$p<q$ using the so-called log bump conditions.  
These results complement the strong type inequalities of
P\'erez~\cite{MR1291534} and answer a conjecture from~\cite{MR2797562}.  For
both sets of results our main tool is a corona decomposition adapted
to fractional averages.
\end{abstract}

\section{Introduction} 
In this paper we prove one and two weight norm inequalities for the
Riesz potentials (also referred to as fractional integral operators):
$$I_\al f(x)=\int_{\R^n}\frac{f(y)}{|x-y|^{n-\al}}, \qquad 0<\al<n.$$
Each of our results is analogous to a corresponding result for
Calder\'on-Zygmund operators, and so our work parallels recent
development in the study of sharp one and two weight norm inequalities
for these operators.  Our results are linked by a common technique
that also originated in the study of singular integrals: a corona
decomposition adapted to the fractional case.

The natural scaling of the operator $I_\al$ shows that if
$I_\al:L^p(\R^n)\ra L^q(\R^n)$, then $p$ and $q$ must satisfy the Sobolev
relationship
\begin{equation}\label{eqn:sobolev}
\frac1p-\frac1q=\frac{\al}{n},
\end{equation}
and so this is a natural condition to assume when studying one weight
inequalities.    
Muckenhoupt and Wheeden~\cite{muckenhoupt-wheeden74}
proved the following result.  

\begin{theorem} \label{thm:MW-old}
Given $\alpha$, $0<\alpha<n$ and $p$, $1<p<n/\alpha$, define $q$ by
\eqref{eqn:sobolev}.   Then the following are equivalent:
\begin{enumerate}

\item $w\in A_{p,q}$:  
\[ [w]_{A_{p,q}} = \sup_Q \left(\avgint_Q w(x)^q\,dx \right)^{1/q}
\left(\avgint_Q w(x)^{-p'}\,dx \right)^{1/p'} < \infty; \]

\item $I_\alpha$ satisfies the weak type inequality
\[ \sup_t t\|w\chi_{\{x : |I_\alpha f(x)|>t\}}\|_q \leq 
C\|fw\|_p; \]

\item $I_\alpha$ satisfies the strong type inequality
 \[ \|(I_\alpha f)w\|_q \leq C\|fw\|_p. \]

\end{enumerate}
\end{theorem}

More recently, the second author with Lacey, P\'erez and
Torres~\cite{MR2652182} proved sharp bounds for these inequalities in
terms of the $[w]_{A_{p,q}}$ constant.  This question was motivated by
the corresponding problem for Calder\'on-Zygmund singular integrals,
which has been studied intensively for more than a decade, and was
recently solved in full generality by Hyt\"onen~\cite{hytonenP2010}.
For the complete history of the problem we also refer the reader to
\cite{dcu-martell-perez,MR2657437, Lern2012, lerner-IMRN2012, PTV2010} and
the references they contain.

The problem of sharp constants for Riesz potentials is more tractable if we reformulate
Theorem~\ref{thm:MW-old} in terms of the Muckenhoupt $A_p$ weights.
Recall that for $1<p<\infty$, $w\in A_p$ if
\[ [w]_{A_p} = \sup_Q \left(\avgint_Q w(x)\,dx\right)
\left(\avgint_Q w(x)^{1-p'}\,dx\right)^{p-1} < \infty. \]
Let $u=w^q$ and $\sigma=w^{-p'}$
and define the function $s(\cdot)$ by 
$$s(p):=1+\frac{q}{p'}=q\Big(1-\frac{\al}{n}\Big)=p\Big(\frac{n-\al}{n-\al p}\Big).$$
Note that it follows at once from this that $s(p)'=s(q')$.  Then it is straightforward to show that the following are equivalent:
$w\in A_{p,q}$, $u\in A_{s(p)}$ and $\sigma \in A_{s(q')}$, and 
\begin{equation*} %\label{eqn:Ap-identity}
 [w]_{A_{p,q}} = [u]_{A_{s(p)}}^{1/q} =
[\sigma]_{A_{s(q')}}^{1/p'}.
\end{equation*}
Moreover, by a change of variables we can restate the weak and strong
type inequalities in terms of $u$ and $\sigma$:
\begin{gather}\label{eqn:weak} 
\|I_\al(f\sigma)\|_{L^{q,\infty}(u)}\lesssim \|f\|_{L^p(\sigma)} \\
\intertext{and}
\label{eqn:strong} 
\|I_\al(f\sigma)\|_{L^{q}(u)}\lesssim \|f\|_{L^p(\sigma)}.
\end{gather}
(This formulation has the advantage that it makes the connection
between the one and two weight inequalities more natural:  see below.)  
It was shown in~\cite{MR2652182} that
\begin{gather}
\label{eqn:shpweak}
\|I_\al(\,\cdot\,\sigma)\|_{L^p(\sigma)\ra L^{q,\infty}(u)}\lesssim
[u]_{A_{s(p)}}^{1-\frac{\al}{n}} \\
\intertext{and}
\label{eqn:shpstrong}
\|I_\al(\,\cdot\,\sigma)\|_{L^p(\sigma)\ra L^{q}(u)}\lesssim
[u]_{A_{s(p)}}^{1-\frac{\al}{n}}+[\sigma]_{A_{s(q')}}^{1-\frac{\al}{n}}\simeq
[u]_{A_{s(p)}}^{(1-\frac{\al}{n})\max(1,\frac{p'}{q})}. 
\end{gather}

Our
first result is an improvement of these inequalities.   It is again
motivated by the corresponding problem for Calder\'on-Zygmund
operators:
see~\cite{hytonen-lacey-IUMJ,hytonen-perez-analPDE,lacey-HJM}.  There,
a precise bound involving the $A_p$ constant and the smaller $A_\infty$ constant was given.  Recall that $w\in
A_\infty$ if
\[ [w]_{A_\infty} = \sup_Q \exp\left(\avgint_Q -\log(w(x))\,dx\right)
\left(\avgint_Q w(x)\,dx\right) < \infty. \]
We have that $w\in A_\infty$ if and only if $w\in A_p$ for some $p>1$,
and 
\[ [w]_{A_\infty} = \lim_{p\rightarrow \infty} [w]_{A_p}. \]
(This limit was proved by Sbordone and Wik~\cite{MR1291957}.)  There
are several equivalent definitions of the $A_\infty$ condition
(see~\cite{garcia-cuerva-rubiodefrancia85}).  One in particular has
been shown to be very useful in the study of sharp constant problems.
We say that a weight $w\in A_\infty'$ if
\[ [w]_{A_\infty'} =\sup_Q \frac{1}{w(Q)}\int_Q M(\chi_Qw)(x)\,dx <
\infty, \]
where $M$ is the Hardy-Littlewood maximal operator.  Independently,
Fujii~\cite{MR0481968} and Wilson~\cite{MR883661,wilson89} (also see~\cite{wilson07})
showed that $w\in A_\infty$ if and only if $w\in A_\infty'$.  P\'erez
and Hyt\"onen~\cite{hytonen-perez-analPDE} showed that
$[w]_{A_\infty'} \lesssim [w]_{A_\infty}$, and in fact 
$[w]_{A_\infty'}$ can be substantially smaller.   Using this
definition we can state our result.

\begin{theorem} \label{thm:mixedfrac-OLD} 
  Given $\alpha$, $0<\alpha<n$, and $p$, $1<p<n/\al$, define $q$ 
  by \eqref{eqn:sobolev}.  Let $w\in A_{p,q}$  and set $u=w^q$ and
  $\sigma=w^{-p'}$.  Then
\begin{gather*}\label{eqn:mixedweak}
\|I_\al(\,\cdot\,\sigma)\|_{L^p(\sigma)\ra L^{q,\infty}(u)}\lesssim
[u]_{A_{s(p)}}^{\frac{1}{q}}[u]_{A_\infty'}^{\frac{1}{p'}} \\
\intertext{and}
\label{eqn:mixedstrong}
\|I_\al(\,\cdot\,\sigma)\|_{L^p(\sigma)\ra L^{q}(u)}\lesssim
[u]_{A_{s(p)}}^{\frac{1}{q}}([u]_{A_\infty'}^{\frac{1}{p'}}
+[\sigma]_{A_\infty'}^{\frac{1}{q}}).
\end{gather*}
\end{theorem}

\begin{remark}
Recently, Lerner~\cite{lerner2012} introduced a different approach to
improving the sharp $A_p$ estimates for singular integrals using a
mixed $A_p$-$A_r$ condition, $1<r<\infty$.  He showed that this
condition is not readily comparable to the $A_p$-$A_\infty$ condition
we are using.  It is an open question whether the corresponding
conditions can be proved for Riesz potentials.
\end{remark}

We will actually prove Theorem~\ref{thm:mixedfrac-OLD} as a special
case of a two weight result.   Note that while we have assumed in
inequalities~\eqref{eqn:weak} and~\eqref{eqn:strong} that $u$ and
$\sigma$ are linked via
the weight $w\in A_{p,q}$, we do not {\em a priori} have to assume
this.  We cannot completely decouple the weights $u$ and $\sigma$ but
we can weaken their connection.  In this context it is natural to
generalize the $A_p$ condition to hold for a pair of weights:  we say
$(u,\sigma)\in A_p$ if
$$[u,\sigma]_{A_p}=\sup_Q \left(\avgint_Q u(x)\,dx\right)
\left(\avgint_Q \sigma(x)\,dx\right)^{p-1}<\infty.$$
It is well known that this condition is necessary for many two weight
inequalities, but not sufficient.   For example, $(u,\sigma)\in A_p$
is necessary for~\eqref{eqn:weak}.  However, if we assume that $u$
and/or $\sigma$ are in $A_\infty$, then it is sufficient, and we can
generalize Theorem~\ref{thm:mixedfrac-OLD} as follows.  

\begin{theorem} \label{thm:mixedfrac} 
  Given $\alpha$, $0<\alpha<n$, and $p$, $1<p<n/\al$, define $q$ 
  by \eqref{eqn:sobolev}.  Suppose $(u,\sigma)$ is a pair of weights
  with $[u,\sigma]_{A_{s(p)}}<\infty$.  If $u\in A_\infty$,  then
\begin{gather*}\label{eqn:mixedweak}
\|I_\al(\,\cdot\,\sigma)\|_{L^p(\sigma)\ra L^{q,\infty}(u)}\lesssim
[u,\sigma]_{A_{s(p)}}^{\frac{1}{q}}[u]_{A_\infty'}^{\frac{1}{p'}}. \\
\intertext{Moreover, if both $u$ and $\sigma$ belong to $A_\infty$, then}
\label{eqn:mixedstrong}
\|I_\al(\,\cdot\,\sigma)\|_{L^p(\sigma)\ra L^{q}(u)}\lesssim
[u,\sigma]_{A_{s(p)}}^{\frac{1}{q}}([u]_{A_\infty'}^{\frac{1}{p'}}
+[\sigma]_{A_\infty'}^{\frac{1}{q}}).
\end{gather*}
\end{theorem}

\begin{remark}
To see that Theorem \ref{thm:mixedfrac} does indeed generalize
Theorem~\ref{thm:mixedfrac-OLD}, set $u=w^q$ and $\sigma=w^{-p'}$.
Then
$[u,\sigma]_{A_{s(p)}}^{\frac1q}=[\sigma,u]_{A_{s(q')}}^{\frac{1}{p'}}$
and
$\frac{1}{q}+\frac{1}{p'}=1-\frac{\al}{n}.$
\end{remark}

\smallskip

A non-quantitative version of this result was implicit in P\'erez~\cite{MR1291534}.
In the study of two weight norm inequalities for singular integrals,
it has long been part of the folklore that assuming $(u,\sigma)\in A_p$  with the
additional hypothesis that $u$ and $\sigma$ are in
$A_\infty$ is a sufficient condition.  This was implicit in
Neugebauer~\cite{MR687633} and was the motivation for results by
Fujii~\cite{fujii91}, Leckband~\cite{leckband85}, and
Rakotondratsimba~\cite{rakotondratsimba98b}.   The sharp analog of
Theorem~\ref{thm:mixedfrac} for singular integrals is due to Hyt\"onen
and Lacey~\cite{hytonen-lacey-IUMJ}.

If we drop the assumption that $u$ and $\sigma$ are $A_\infty$
weights, we need to assume a stronger condition than two weight $A_p$
for norm inequalities to hold.  However, when working in this
generality we no longer have to assume that $p$ and $q$ satisfy the
Sobolev relationship~\eqref{eqn:sobolev}.  Instead, we only assume
that $p\leq q$.  (The case $q<p$ is much more difficult; see, for
instance, Verbitsky~\cite{MR1134691}.)  In this case the weights for
the weak and strong type inequalities were characterized by
Sawyer~\cite{MR719674,MR930072}.

\begin{theorem} \label{thm:sawyer}
Given $\alpha$, $0<\alpha<n$, and $p,\,q$, $1<p\leq q < \infty$, the
weak type inequality~\eqref{eqn:weak} holds if and only if for every
cube $Q$,
\begin{equation} \label{eqn:sawyer2}
  \left(\int_Q I_\alpha(\chi_Qu)(x)^{p'} \sigma(x)\,dx\right)^{1/p'} 
\lesssim \left(\int_Q u(x)\,dx\right)^{1/q'}.
\end{equation}
The strong type inequality~\eqref{eqn:strong} holds if and only if for
every cube $Q$, inequality~\eqref{eqn:sawyer2} holds and 
\begin{equation} \label{eqn:sawyer1}
  \left(\int_Q I_\alpha(\chi_Q\sigma)(x)^q u(x)\,dx\right)^{1/q} 
\lesssim \left(\int_Q \sigma(x)\,dx\right)^{1/p}.
\end{equation}
\end{theorem}

While  necessary and sufficient, the so-called testing conditions in
Theorem~\ref{thm:sawyer} have the drawback that they involve the Riesz
potential itself.  Another approach is to find sharp sufficient conditions that
resemble the $A_{p,q}$ condition of Muckenhoupt and Wheeden.   This
approach was introduced by P\'erez~\cite{MR1291534,MR1327936} and
involves replacing the local $L^p$ norm with a larger norm in the
scale of Orlicz spaces.  

To state these results we need to make some preliminary definitions. A
Young function is a function $\Phi : [0,\infty) \rightarrow
[0,\infty)$ that is continuous, convex and strictly increasing,
$\Phi(0)=0$ and $\Phi(t)/t\rightarrow \infty$ as $t\rightarrow
\infty$. Given a cube $Q$ we define the localized Luxemburg norm by
\[  \|f\|_{\Phi,Q} = 
\inf \left\{ \lambda > 0 : \avgint_Q \Phi\left(\frac{|f(x)|}{\lambda}\right)dx \leq 1 \right\}.  \]
When $\Phi(t)=t^p$, $1<p<\infty$, this becomes the $L^p$ norm and we write
$$
\|f\|_{p,Q}=\left(\avgint_Q |f(x)|^p dx\right)^{1/p}.
$$
The associate function of $\Phi$ is the Young function
\[ \bar{\Phi}(t) = \sup_{s>0}\{ st - \Phi(s)\}. \]
%n
Note that $\bar{\bar{\Phi}}=\Phi$.
A Young function $\Phi$ satisfies the $B_p$ condition if for some $c>0$,
\[ \int_c^\infty \frac{\Phi(t)}{t^p} \frac{dt}{t} < \infty. \]
Important examples of such functions are $\Phi(t)=t^{sp}$, $s>1$, whose
associate function is $\bar{A}(t)=t^{(sp)'}$, and
$\Phi(t)=t^{p}\log(e+t)^{-1-\epsilon}$, $\epsilon>0$, which have
associate functions $\bar{\Phi}(t)\approx
t^{p'}\log(e+t)^{p'-1+\delta}$, $\delta>0$.  We refer to these
associate functions as power bumps and log bumps.

P\'erez proved the following strong type inequality.

\begin{theorem} \label{thm:perez-strong}
Given $\alpha$, $0<\alpha<n$, and $p,\,q$, $1<p\leq q < \infty$, the
strong type inequality~\eqref{eqn:strong} holds for every pair of
weights $(u,\sigma)$ that satisfies
\begin{equation} \label{eqn:perez1}
 \sup_Q |Q|^{\frac{\alpha}{n}+\frac{1}{q}-\frac{1}{p}}
\|u^{\frac{1}{q}}\|_{\Phi,Q} \|\sigma^{\frac{1}{p'}}\|_{\Psi,Q} <
\infty,
\end{equation}
where $\Phi,\,\Psi$ are Young functions such that $\bar{\Phi}\in
B_{q'}$  and $\bar{\Psi} \in B_p$. 
\end{theorem}

The corresponding two weight result for singular integrals (with $p=q$
and $\alpha=0$) was a long-standing conjecture motivated by
Theorem~\ref{thm:perez-strong}.  It was recently proved by
Lerner~\cite{Lern2012}.  For a detailed history of this problem,
see~\cite{dcu-martell-perez,MR2797562} and the references they
contain.

Much less is known about two weight, weak type inequalities for the
Riesz potential.   It has
long been known that for singular integrals, a sufficient condition
for the weak $(p,p)$ inequality is that the weights satisfy
\[ \sup_Q \|u^{\frac{1}{p}}\|_{\Phi,Q}\|\sigma^{\frac{1}{p'}}\|_{p',Q}
< \infty, \]
where $\Phi$ is the log bump $\Phi(t)=t^p\log(e+t)^{p-1+\delta}$
(see~\cite{MR1713140}).  It is conjectured that it
suffices to take $\Phi\in B_{p'}$ (see~\cite{MR2797562}.)  Moreover,
it was conjectured that the corresponding result holds for Riesz
potentials.

\begin{conjecture} \label{conj:weaktype} 
Given $\alpha$, $0<\al<n$, and $p,\,q$, $1< p\leq q<\infty$,  then the
weak type inequality~\eqref{eqn:weak} holds for every pair of weights
$(u,\sigma)$ that satisfies
\[ \sup_Q |Q|^{\frac{\alpha}{n}+\frac{1}{q}-\frac{1}{p}}
\|u^{\frac{1}{q}}\|_{\Phi,Q} \|\sigma^{\frac{1}{p'}}\|_{p',Q} <
\infty, \]
where $\Phi$ is a  Young function such that $\bar{\Phi}\in
B_{q'}$.
\end{conjecture}

Until now, Conjecture~\ref{conj:weaktype} was only known when $\Phi$
is power bump (see~\cite{MR1860236,MR1793688}) or a log bump of the
form $\Phi(t) =t^q\log(e+t)^{2q-1+\delta}$.  (This is proved
in~\cite{MR2797562} when $p=q$, but the same proof works in the
case $q>p$.)  In the scale of log bumps the conjecture should hold with
the smaller exponent $q-1+\delta$.  Our first result is a proof of
this for a limited range of values of $p$ and $q$.

\begin{theorem} \label{thm:mainweak}
Given $\alpha$, $0<\al<n$, and $p,\,q$, $1< p\leq q<\infty$,  suppose
\begin{equation} \label{eqn:mainweak1}
 \frac{p'}{q'}\left(1-\frac{\alpha}{n}\right) \geq 1.
\end{equation}
Then the
weak type inequality~\eqref{eqn:weak} holds for every pair of weights
$(u,\sigma)$ that satisfies
\[ \sup_Q |Q|^{\frac{\alpha}{n}+\frac{1}{q}-\frac{1}{p}}
\|u^{\frac{1}{q}}\|_{\Phi,Q} \|\sigma^{\frac{1}{p'}}\|_{p',Q} <
\infty, \]
where $\Phi(t)=t^q\log(e+t)^{q-1+\delta}$, $\delta>0$.  
\end{theorem}

The restriction \eqref{eqn:mainweak1} holds when $p$ and $q$ satisfy
the Sobolev relationship~\eqref{eqn:sobolev}, and it also holds for
$p$ and $q$ close to these values.  It does not hold, however, when
$p=q$.  This condition appears to be intrinsic to our proof and a
new approach will be necessary to prove Theorem~\ref{thm:mainweak} for
the full range of $p$ and $q$.  

\begin{remark}
  The proof of the corresponding result for singular integrals is much
  simpler than the proof of Theorem~\ref{thm:mainweak}: it follows by
  extrapolation from a two weight, weak $(1,1)$ inequality for
  singular integrals.  It is conjectured that a similar inequality
  holds for Riesz potentials, and this would yield a simpler proof of
  Theorem~\ref{thm:mainweak}.  See~\cite{MR2797562} for complete
  details.
\end{remark}

\begin{remark}
The weak type results for singular integral operators are sharp in the sense
that they are false if we take $\delta=0$ in the definition of
$\Phi$ (see~\cite{MR1713140}).  Though it has not appeared explicitly in the
  literature, the same is true for Riesz potentials.  For an example
involving commutators of Riesz potentials,
see~\cite{cruz-moen2012}. 
\end{remark}

By a small modification of our proof of Theorem~\ref{thm:mainweak} we
can extend this result to a class of Young functions referred to as
loglog bumps (cf.~\cite{MR2797562}).  Our proof builds upon the recent
work in~\cite{CRV2012}, where a weak type inequality for singular
integrals involving loglog bumps was proved.

\begin{theorem} \label{thm:weakloglog}
With the same hypotheses as before, the conclusion of
Theorem~\ref{thm:mainweak} remains true if 
 $\Phi(t)=t^q\log(e+t)^{q-1}\log\log(e^e+t)^{q-1+\delta}$ for
 $\delta>0$  sufficiently large.
\end{theorem}

In the scale of loglog bumps, Conjecture~\ref{conj:weaktype} holds for
loglog bumps if we take any $\delta>0$.   But again this restriction
on $\delta$ seems to be intrinsic to our proof.

\bigskip

Our second result in this vein is an improvement of
Theorem~\ref{thm:perez-strong} in the scale of log bumps.  We
believe that the single condition~\eqref{eqn:perez1} with a bump on
each term can be replaced by two conditions, each with a single bump
condition.  This is referred to as a separated bump condition.  More
precisely, we make the following conjecture.

\begin{conjecture} \label{conj:strong}
Given $\alpha$, $0<\al<n$, and $p,\,q$, $1< p\leq q<\infty$, then the
strong type inequality~\eqref{eqn:strong} holds for every pair of weights
$(u,\sigma)$ that satisfies
\begin{gather*}
\sup_Q |Q|^{\frac{\al}{n}+\frac1q-\frac1p}\|u^{\frac1q}\|_{\Phi,Q}
\|\sigma^{\frac{1}{p'}}\|_{p',Q} <\infty, \\
\sup_Q |Q|^{\frac{\al}{n}+\frac1q-\frac1p} \|u^{\frac{1}{q}}\|_{q,Q}
\|\sigma^{\frac{1}{p'}}\|_{\Psi,Q}<\infty,
\end{gather*}
where $\Phi,\,\Psi$ are Young functions such that $\bar{\Phi}\in
B_{q'}$  and $\bar{\Psi} \in B_p$. 
\end{conjecture}

The motivation for this conjecture is recent work on two weight norm
inequalities for singular integrals.  The corresponding conjecture for
singular integrals has been implicit in the literature, as it is
closely connected to a long-standing conjecture of Muckenhoupt and
Wheeden, now known to be false.  It was recently made explicit
in~\cite{CRV2012}; this paper also discusses its connection with
the Muckenhoupt-Wheeden conjecture.  Moreover, the authors also proved
the conjecture in the special case of log bumps and certain loglog
bumps.  We can prove these kinds of result for Riesz potentials.

\begin{theorem} \label{thm:mainstrong} 
Given $\alpha$, $0<\al<n$, and $p,\,q$, $1< p\leq q<\infty$,  suppose
\begin{equation}\label{eqn:rangepq2} 
\min\Big(\frac{q}{p},\frac{p'}{q'}\Big)(1-\frac{\al}{n})\geq 1.
\end{equation} 
Then the
strong type inequality~\eqref{eqn:strong} holds for every pair of weights
$(u,\sigma)$ that satisfies
\begin{gather*}
\sup_Q |Q|^{\frac{\al}{n}+\frac1q-\frac1p}\|u^{\frac1q}\|_{\Phi,Q}
\|\sigma^{\frac{1}{p'}}\|_{p',Q} <\infty, \\
\sup_Q |Q|^{\frac{\al}{n}+\frac1q-\frac1p} \|u^{\frac{1}{q}}\|_{q,Q}
\|\sigma^{\frac{1}{p'}}\|_{\Psi,Q}<\infty,
\end{gather*}
where 
$\Phi(t)=t^q\log(e+t)^{q-1+\delta}$ and $\Psi(t)=t^{p'}\log(e+t)^{p'-1+\delta}$.
\end{theorem}

\begin{theorem} \label{thm:strongloglog} 
With the same hypotheses as before, the conclusion of
Theorem~\ref{thm:mainstrong} remains true if 
\begin{gather*}
\Phi(t)=t^q\log(e+t)^{q-1}\log\log(e^e+t)^{q-1+\delta}\\ 
\Psi(t)=t^{p'}\log(e+t)^{p'-1}\log\log(e^e+t)^{p'-1+\delta}
\end{gather*}
for $\delta>0$  sufficiently large.
\end{theorem}

Similar to the restriction in Theorem~\ref{thm:mainweak}, 
\eqref{eqn:rangepq2} includes $p$ and $q$ that satisfy the Sobolev
relationship~\eqref{eqn:sobolev} but does not extend to include the
case $p=q$.

\medskip

Finally, as an application of our weak type results we can prove a two weight,
Sobolev inequality.     Such inequalities follow immediately from our
strong type results and the well-known inequality
$$|f(x)|\lesssim I_1(|\nabla f|)(x).$$
However, by the truncation method of Maz'ya \cite{MR2777530} (see also
\cite{MR1683160,MR1187073}), a strong type inequality for the gradient
can be deduced from a weak type inequality for the Riesz potential.
The following corollary to Theorem~\ref{thm:mainweak} can be proved
exactly as~\cite[Theorem~2.7]{MR2652182}.  (See also~\cite[Lemma~4.31]{MR2797562}.)

\begin{corollary} \label{cor:sobolevest} 
Given $p,\,q$, $1< p\leq q<\infty$,  suppose $\frac{p'}{q'}\geq n'$.
Then for all smooth functions $f$ with compact support, 
$$\left(\int_{\R^n}|f(x)|^q u(x)\,dx\right)^{1/q}\lesssim 
\left(\int_{\R^n}|\nabla f(x)|^pv(x)\,dx\right)^{1/p}$$
for all pairs of weights $(u,v)$  that satisfy
$$\sup_Q |Q|^{\frac{1}{n}+\frac1q-\frac1p}\|u^{\frac1q}\|_{\Phi,Q}\|v^{-1/p}\|_{p',Q}
<\infty,$$
where  $\Phi(t)=t^q\log(e+t)^{q-1+\delta}$.
\end{corollary}

\subsection*{Organization}
The remainder of this paper is organized as follows.  In
Section~\ref{preliminaries} we gather some results about dyadic
operators that are used in our proofs.  In particular, we state a
sharp dyadic version of Theorem~\ref{thm:sawyer}.   In
Section~\ref{sec3} we prove Theorem \ref{thm:mixedfrac}, and in
Section~\ref{proofs} we prove Theorems \ref{thm:mainweak},
\ref{thm:weakloglog}, \ref{thm:mainstrong}
and~\ref{thm:strongloglog}.  

Throughout the paper, all of the notation we will use will be standard
or defined as needed.  All cubes in $\R^n$ will assume to be half
open with sides parallel to the axes.  Given a cube, $Q$, $\ell(Q)$
will denote its side length.  Given a set $E\subseteq \R^n$, $|E|$
will denote the Lebesgue measure of $E$, $w(E)=\int_E w\,dx$ the
weighted measure of $E$, and $\dashint_E w\,dx={|E|^{-1}}\int_E
w\,dx=w(E)/|E|$ the average of $w$ over $E$.    In proving
inequalities, if we write $A \lesssim B$, we mean that $A \leq CB$,
where the constant $C$ can depend on $\alpha$, $p$ and $n$, but does
not depend on the weights $u$ or $\sigma$, nor on the function.   If
we write $A\simeq B$, then $A\lesssim B$ and $B\lesssim A$.

\section{Dyadic Riesz potentials}
\label{preliminaries}

In this section we define two dyadic versions of the Riesz potential,
and show how these can be used to approximate $I_\alpha$.  We begin by
defining special collections of cubes, known as {\it dyadic
  grids} or {\it filtrations}.  A dyadic grid $\D$ 
is a countable collection of cubes that has the following properties:
\begin{enumerate}[(1)]
\item  $Q\in \D$ $\Rightarrow$ $\ell(Q)=2^k$ for some $k\in \Z$,
\item $Q,P\in \D$ $\Rightarrow$ $Q\cap P\in \{\varnothing,P,Q\}$,
\item and for each $k\in \Z$ the set $\D_k=\{Q\in \D:\ell(Q)=2^k\}$ forms a partition of $\R^n$. 
\end{enumerate}
The collection of dyadic cubes used to form the well-known Calder\'on-Zygmund
decomposition are a dyadic grid, as are all of the translates of these
cubes.  Below we will make extensive use of the dyadic grids
$$\D^t=\{2^{-k}([0,1)^n+m+(-1)^kt):k\in \Z, m\in \Z^d\}, \qquad t\in \{0,1/3\}^n.$$
The importance of these grids is shown by the following proposition;
a proof can be found in~\cite{Lern2012}.

\begin{prop} \label{prop:shiftdyad} 
  Given any cube $Q$ in $\R^n$ there exists a $t\in \{0,1/3\}^n$
  and a cube $Q_t\in \D^t$ such that $Q\subseteq Q_t$ and
  $\ell(Q_t)\leq 6\ell(Q)$.
\end{prop}

Given a dyadic grid, $\D$, define the dyadic Riesz potential operator
\begin{equation}\label{dyadicfracint}
I_\al^\D f(x)=\sum_{Q\in \D}|Q|^{\frac{\al}{n}}\dashint_Q f(y)\,dy\cdot
\chi_Q(x).
\end{equation}
Dyadic Riesz potentials were first introduced by Sawyer and
Wheeden~\cite{MR1291534} (see also~\cite{MR1175693}).  They proved
(essentially) that the Riesz potential lies in the convex hull of all the dyadic
Riesz potentials.    Here we prove a sharper version of this result. 

\begin{prop} \label{dyadicbound} Given $0<\al<n$ and a non-negative
  function $f$,  then for any dyadic grid $\D$,
$$ I^\D_\al f(x)\lesssim I_\al f(x). $$
Conversely, we have that
\[  I_\alpha f(x)  \lesssim \max_{t\in \{0,1/3\}^n}
I_\al^{\D^t}f(x). \] 
\end{prop}

Note that as a corollary to Proposition~\ref{dyadicbound} we have that
$I_\alpha f$ is pointwise equivalent to a linear combination of dyadic
Riesz potentials:
\[ I_\alpha f(x) \simeq \sum_{t\in \{0,1/3\}^n} I^{\D^t}_\alpha
f(x). \]

\begin{proof}
  Fix a non-negative function $f$, $x\in \R^n$, and a dyadic grid
  $\D$.  Let $\{Q_k\}_{k\in \Z}$ be the unique sequence of dyadic
  cubes in $\D$ such that $\ell(Q_k)=2^k$ and $x\in Q_k$.  Fix $N\geq 1$;  then
\begin{align*}
& \sum_{\substack{Q\in \D\\ \ell(Q)\leq 2^N}}
  \frac{1}{|Q|^{1-\frac{\al}{n}}}\int_Q f(y)\,dy\cdot \chi_Q(x) \\
& \qquad \qquad =\sum_{k=-\infty}^N\frac{1}{|Q_k|^{1-\frac{\al}{n}}}\int_{Q_k} f(y)\,dy\\
& \qquad \qquad =\sum_{k=-\infty}^N\frac{1}{|Q_k|^{1-\frac{\al}{n}}}
\int_{Q_k\backslash Q_{k-1}} f(y)\,dy
+\sum_{k=-\infty}^N\frac{1}{|Q_k|^{1-\frac{\al}{n}}}\int_{Q_{k-1}} f(y)\,dy\\
& \qquad \qquad \lesssim \sum_{k=-\infty}^N
\int_{Q_k\backslash Q_{k-1}}\frac{f(y)}{|x-y|^{n-\al}}\,dy
+2^{\al-n}\sum_{\stackrel{Q\in \D}{\ell(Q)\leq 2^N}} 
\frac{1}{|Q|^{1-\frac{\al}{n}}}\int_Q f(y)\,dy\cdot \chi_Q(x)\\
& \qquad \qquad =\int_{Q_N} \frac{f(y)}{|x-y|^{n-\al}} \,dy
+ 2^{\al-n}\sum_{\substack{Q\in \D\\ \ell(Q)\leq 2^N}} 
\frac{1}{|Q|^{1-\frac{\al}{n}}}\int_Q f(y)\,dy\cdot \chi_Q(x).
\end{align*}
Since $\alpha<n$ we can rearrange terms  and take the limit as $N\ra
\infty$ to get
$$I^{\D}_\al f(x)\lesssim I_\al f(x).$$

To prove the second inequality, let $Q(x,r)$ be the cube of
side-length $2r$ centered at $x$.  By standard estimates (see, for
example~\cite{MR2652182}),
\begin{align*}
I_\al f(x)&\leq 2^{n-\al}\sum_{k\in \Z} (2^{-k})^{n-\al}\int_{Q(x,2^k)} f(y)\,dy.
\end{align*}
By Proposition \ref{prop:shiftdyad}, for each $k\in \Z$ there exists
$t\in \{0,1/3\}^n$ and $Q_t\in \D^t$ such that $Q(x,2^k)\subset Q_t$
and
\[  2^{k+1}=\ell(Q)\leq \ell(Q_t)\leq 6\ell(Q(x,2^k))=12\cdot 2^k.\]
Since $\ell(Q_t)=2^j$ for some $j$, we must have that  $2^{k+1}\leq
\ell(Q_t)\leq 2^{k+3}$.  Hence, 
\begin{align*}
I_\al f(x)&\leq 2^{n-\al}\sum_{k\in \Z} (2^{-k})^{n-\al}\int_{Q(x,2^k)} f(y)\,dy \\
&\lesssim \sum_{k\in \Z} \sum_{t\in \{0,1/3\}^n} \sum_{\substack{Q\in \D^t\\ 2^{k+1}\leq \ell(Q)\leq 2^{k+3}}}  \frac{1}{|Q|^{1-\frac{\al}{n}}}\int_Q f(y)\,dy\cdot \chi_Q(x)\\
&\lesssim \sum_{t\in \{0,1/3\}^n}  \sum_{Q\in \D^t}  
\frac{1}{|Q|^{1-\frac{\al}{n}}}\int_Q f(y)\,dy\cdot \chi_Q(x) \\
& \lesssim \max_{t\in \{0,1/3\}^n} I_\al^{\D^t}f(x). 
\end{align*}
\end{proof}

We now show that in the definition of $I_\alpha^\D$ we can replace the
summation over $\D$ by a summation over a subset of the dyadic grid
whose members have good intersection properties.  We call such a subset a {\it
  sparse family} (cf.~\cite{HyNaz2012,Lern2012}).  Given a dyadic grid
$\D$, a subset $\Sp\subseteq \D$  is a sparse family
of dyadic cubes if for every $Q\in \Sp$,
\begin{equation}\label{eqn:sparse}
\Big|\bigcup_{\substack{Q'\in \Sp\\ Q'\subsetneq Q}} Q'\Big|\leq
\frac{1}{2}|Q|.
\end{equation}
 If $\Sp$ is a sparse family and we define the sets
$$E(Q)=Q\backslash \Big(\bigcup_{\substack{Q'\in \Sp \\ Q'\subsetneq Q}}
Q'\Big),  \qquad Q\in \Sp,$$
then the collection $\{E(Q)\}_{Q\in \Sp}$ is pairwise disjoint and for
each $Q$,
\begin{equation}\label{eqn:equivmeasure}
|E(Q)|\leq |Q|\leq 2|E(Q)|.
\end{equation}
Though the terminology is recent, particular sparse families have long
played a role in the applications of  Calder\'on-Zygmund theory.   See,
for example,~\cite[Chapter~4, Lemma~2.5]{garcia-cuerva-rubiodefrancia85} or~\cite[Appendix~A]{MR2797562}.

Given $\alpha$, $0<\alpha<n$, and a sparse family $\Sp\subseteq \D$,
define the sparse dyadic Riesz potential
$$I^\Sp_\al f(x)=
\sum_{Q\in \Sp} |Q|^{\frac{\al}{n}}\dashint_Q f\,dy\cdot \chi_Q(x).$$
The connection between dyadic Riesz potentials and their sparse
counterparts is given by the following result.  The ideas underlying
the proof are not new:  they are implicit in
\cite{MR2652182,MR1291534,MR1175693}.

\begin{prop} \label{prop:sparsefrac} 
Given a bounded, non-negative function $f$ with compact
  support and a dyadic grid $\D$,  there exists a sparse family $\Sp$ such
  that for all $\alpha$, $0<\alpha<n$,
$$I^\D_\al f(x)\lesssim I^\Sp_\al f(x).$$ 
\end{prop}

\begin{proof} Let $a=2^{n+1}$.  For each $k\in \Z$ define 
$$\Q^k=\Big\{P\in \D:a^k<\avgint_P f\,dy\leq a^{k+1}\Big\}.$$
Then for every $P\in \D$ such that $\avgint_P f\,dy \neq 0$, there
exists a unique $k$ such that $P\in \Q^k$.  Therefore,
\begin{multline*}
I^\D_\al f(x)=\sum_{P\in \D}  \frac{1}{|P|^{1-\frac{\al}{n}}}\int_{P} f\,dy\cdot \chi_{P}(x) \\
=\sum_k \sum_{P\in \Q^k} \frac{1}{|P|^{1-\frac{\al}{n}}}\int_{P} f\,dy\cdot \chi_{P}(x) 
\leq \sum_k a^{k+1} \sum_{P\in \Q^k}{|P|^{\frac{\al}{n}}}\cdot \chi_{P}(x).
\end{multline*}
Now let $\Sp_k$ be the collection of disjoint, maximal cubes $Q\in \D$ such that
$$\dashint_Q f\,dx>a^k.$$
(Such a collection exists since $\D$ is a dyadic grid and $f$ is
bounded and has compact support.)  Let $\Sp=\bigcup_k \Sp_k$.  Then
for every $P\in \Q^k$ there exists $Q\in \Sp_k$ such that $Q\supseteq
P$.  Hence, we have that
$$ I^\D_\al f(x)\leq a\sum_{k} a^k\sum_{Q\in \Sp_k} 
\sum_{\substack{P\in \D \\ P\subseteq Q}} |P|^{\frac{\al}{n}} \cdot \chi_P(x).$$
The inner sum can be evaluated:
\[ 
\sum_{\substack{P\in \D \\ P\subseteq Q}} |P|^{\frac{\al}{n}}
\cdot\chi_P(x)
=\sum^\infty_{r=0}\sum_{\substack{P\in \D:P\subset
    Q \\ \ell(P)=2^{-r}\ell(Q)}}
|P|^{\frac{\al}{n}}\cdot\chi_P(x)
=\frac{1}{1-2^{-\al}}|Q|^{\frac{\al}{n}}\cdot\chi_{Q}(x). 
\]
Moreover, since $a^k<\avgint_{Q} f\,dy$ if $Q\in \Sp_k$,  we have that
$$I^\D_\al f(x)\lesssim I_\al^\Sp f(x).$$  

Finally, we show that $\Sp$ is sparse.  If $Q\in \Sp$, then $Q\in \Sp_k$
for some $k\in \Z$; hence, by the maximality of the cubes in $\Sp$, 
\[ 
\Big|\bigcup_{\substack{Q'\in \Sp\\ Q'\subsetneq Q}}Q'\Big|
=\sum_{\substack{Q'\in S^{k+1}\\ Q'\subseteq Q}}|Q'|<
\frac{1}{a^k}\sum_{\substack{Q'\in \Sp_{k+1} \\ Q'\subseteq Q}}\int_{Q'} f\,dx
\leq \frac{1}{a^k}\int_Qf\,dx\leq \frac{2^n}{a}|Q|=\frac12|Q|.
\]
\end{proof}

As a consequence of Propositions~\ref{dyadicbound}
and~\ref{prop:sparsefrac}, to prove our main results it will suffice
to work with a general dyadic grid $\D$ and a sparse Riesz potential
$I^\Sp_\al$.  To prove bounds for $I_\al^\Sp$ we
will use a dyadic version of Theorem~\ref{thm:sawyer} due to Lacey,
Sawyer, and Uriarte-Tuero \cite{LacSawUT2010} that gives
precise bounds in terms of testing conditions.    To state their
result, we need a definition.  Given a dyadic grid $\D$ and $R\in \D$,
let
$$I_\al^{\Sp(R)} f(x)=\sum_{\substack{Q\in \Sp \\ Q\subseteq R}} 
|Q|^{\frac{\al}{n}}\dashint_Q f\,dx\cdot\chi_Q(x).$$
For $1<p\leq q<\infty$ and a pair of weights $(u,\sigma)$ define
$$[u,\sigma]^\D_{(I^\Sp_{\al})^{p,q}} = \sup_{R\in \D} 
\sigma(R)^{-1/p}\bigg(\,\int_R I^{\Sp(R)}_\al(\chi_R\sigma)^qu\,dx\bigg)^{1/q}.$$

\begin{prop} \label{prop:lsut}
Fix $\alpha$, $0< \al<n$,  and $p,\,q$, $1<p\leq q<\infty$.  Let  $\D$
be a dyadic grid and let $\Sp$ be a sparse subset of $\D$.  Given any
pair of weights $(u,\sigma)$,  the following equivalences hold:
\begin{gather*}\label{eqn:sparseweak} 
\|I_\al^\Sp(\,\cdot\,\sigma)\|_{L^p(\sigma)\ra L^{q,\infty}(u)}
\simeq [\sigma,u]^\D_{(I_\al^\Sp)^{q',p'}} \\
\label{eqn:sparsestrong} 
\|I_\al^\Sp(\,\cdot\,\sigma)\|_{L^p(\sigma)\ra L^{q}(u)}
\simeq[u,\sigma]^\D_{(I^\Sp_\al)^{p,q}}+ [\sigma,u]^\D_{(I_\al^\Sp)^{q',p'}}.
\end{gather*}
\end{prop}

\section{Proof of Theorem~\ref{thm:mixedfrac}} \label{sec3}

Our main result in this section is the following.  

\begin{theorem} \label{thm:maintest} 
  Given $\alpha$, $0< \al<n$, and $p$, $1<p<n/\al$, define $q$ by
  \eqref{eqn:sobolev}.  Suppose $(u,\sigma)$ is a pair of weights with
  $[u,\sigma]_{A_{s(p)}}<\infty$, $\D$ is a dyadic grid with sparse
  subset $\Sp$.  If $u\in A_\infty$, then
\begin{equation}\label{eqn:mainest2sup} 
[\sigma,u]^\D_{(I^\Sp_\al)^{q',p'}}\lesssim
[u,\sigma]_{A_{s(p)}}^{\frac1q}[u]_{A_\infty'}^{\frac{1}{p'}}. \end{equation}
The constant in \eqref{eqn:mainest2sup}
is independent of $\D$ and $\Sp$.
\end{theorem}

The operator $I_\al^\Sp$ is self adjoint; hence, by symmetry we also have the dual testing condition
$$[u,\sigma]^\D_{(I_\al^\Sp)^{p,q}}\lesssim
[\sigma,u]_{A_{s(q')}}^{\frac{1}{p'}}[\sigma]_{A_\infty'}^{\frac{1}{q}}$$
provided $\sigma\in A_\infty$.
By Propositions~\ref{prop:shiftdyad}, \ref{prop:sparsefrac}
and~\ref{prop:lsut}, Theorem~\ref{thm:mixedfrac} follows at once from Theorem~\ref{thm:maintest}.

\smallskip

The proof of Theorem~\ref{thm:maintest} requires three lemmas.  To
state the first we define the fractional maximal operator with respect to
a dyadic grid $\D$.   Given $\alpha$, $0<\alpha<n$, and a non-negative
measure $\mu$ on $\R^n$ define
\[ M^\D_{\al,\mu}f(x)
=\sup_{Q\in \D} \frac{1}{\mu(Q)^{1-\frac{\al}{n}}}\int_Q |f|\,d\mu
\cdot \chi_Q(x).  \]

\begin{lemma}\label{weightedmax} 
Given $\alpha$, $0< \al<n$, and $p$, $1<p<n/\al$, define $q$ by
  \eqref{eqn:sobolev}.  If  the measure $\mu$ is such that $\mu(\R^n)=\infty$, then
$M^\D_{\al,\mu}:L^p(\mu)\ra L^q(\mu)$.  If $p=1$, then 
$M^\D_{\al,\mu}:L^1(\mu)\ra L^{q,\infty}(\mu)$.
\end{lemma}
\noindent The proof of Lemma \ref{weightedmax} is standard:
see~\cite{stein93} for $\al=0$ and \cite{MR2534183} for $0<\al<n$.

\smallskip

The second Lemma is a fractional Carleson embedding theorem.   We do
not believe that this result is new; however, we give the short proof because we were
unable to find it in the literature.

\begin{lemma}\label{Carleson}
Given a dyadic grid $\D$ and a non-negative measure $\mu$ such that
$\mu(\R^n)=\infty$, suppose
$\{c_Q\}_{Q\in \D}$ is a sequence of nonnegative numbers satisfying
$$\sum_{Q\subseteq R} c_Q\leq A\, \mu(R), \qquad R\in \D.$$
Given $\alpha$, $0< \al<n$, and $p$, $1<p<n/\al$, define $q$ by
  \eqref{eqn:sobolev}.  Then for all non-negative functions $f$, 
\[ 
\left(\sum_{Q\in \D} c_Q\cdot
  \bigg(\frac{1}{\mu(Q)^{1-\frac{\al}{n}}}\int_Qf\,d\mu\bigg)^q\right)^{1/q}
\leq A^{1/q}\|M_{\al,\mu}^\D f\|_{L^q(\mu)} 
\lesssim A^{1/q} \|f\|_{L^p(\mu)}.
\]
\end{lemma}

\begin{proof}  
The second inequality follows at once from Lemma~\ref{weightedmax}.
To prove the first, without loss of generality we may assume that $f$
is bounded and has compact support.
Let $(\D,\nu)$ be the measure space with $\nu(Q)=c_Q$, and define
$$a_{\al,\mu}(f,Q)=\frac{1}{\mu(Q)^{1-\frac{\al}{n}}}\int_Q f\,d\mu.$$
Then
$$\sum_{Q\in \D} c_Q\cdot \big(a_{\al,\mu}(f,Q)\big)^q
=q\int_0^\infty \lambda^{q-1}\nu(\{Q\in
\D:a_{\al,\mu}(f,Q)>\lambda\})\,d\lambda.$$
Let $\Omega_\lambda=\{Q\in \D:a_{\al,\mu}(f,Q)>\lambda\}$ and
$\Omega_\lambda^*$ be the set of all maximal (with respect to
inclusion) dyadic cubes $R$ such that $a_{\al,\mu}(f,R)>\lambda$.
Then the cubes in $\Omega_\lambda^*$ are pairwise disjoint, each
$Q\in \Omega_\lambda$ is contained in some $R\in\Omega_\lambda^*$, and
$$\bigcup_{R\in \Omega_\lambda^*} R=\{M_{\al,\mu}^\D f>\lambda\}.$$
Hence, 
$$\nu(\Omega_\lambda)
=\sum_{Q\in \Omega_\lambda}c_Q
\leq \sum_{R\in \Omega_\lambda^*}
\sum_{Q\subseteq R} c_Q\leq A\sum_{R\in \Omega_\lambda^*} \mu(R)
=A\mu(\{M_{\al,\mu}^\D f>\lambda\}),$$
and so
\[ \sum_{Q\in \D} c_Q\cdot (a_{\al,\mu}(f,Q))^q
\leq Aq\int_0^\infty \lambda^{q-1}\mu(\{M_{\al,\mu}^\D
f>\lambda\})d\lambda
=A\|M_{\al,\mu}^\D f\|_{L^q(\mu)}^q.
\]
\end{proof}

The last lemma is a crucial exponential decay estimate in the spirit of
the John-Nirenberg inequality for BMO functions.  Similar estimates
can be found in \cite[Lemma 5.5]{HPTV2010} and \cite[Lemma
3.15]{MR2657437}.   Our proof is simplified because we are able to
take advantage of the sparse family of cubes. 

\begin{lemma} \label{lemma:expdecyLeb}
  Let $\Sp$ be a sparse family of cubes.  For any cube $R_0$ and every $k\geq 1$, 
\begin{equation} \label{eqn:expdecayLeb}
\Big|\Big\{x\in R_0: \sum_{\stackrel{Q\in \Sp}{Q\subseteq R_0}}\chi_Q(x)>k\Big\}\Big|\leq 2^{-k} |R_0|.
\end{equation}
\end{lemma}

\begin{proof} Given $R_0$, set $\Sp(R_0)=\{Q\in \Sp:Q\subseteq R_0\}$.
  Let $\Ps_1(R_0)$ be the collection of all maximal cubes in
  $\Sp(R_0)$. Define $\Ps_{k+1}(R_0)$ inductively to be the collection
  of all $Q\in \Sp(R_0)$ that are maximal with respect to inclusion
  and such that there exists $Q'\in \Ps_k(R_0)$ with $Q\subsetneq Q'$.
 In other words,  $\Ps_{k+1}(R_0)$ is the collection of maximal cubes that are
  properly contained in the members of $\Ps_{k}(R_0)$.  We will refer
  to the members of $\Ps_k$ as ``cubes at the $k$-th level down."  Let
$$\Omega_k=\bigcup_{Q\in \Ps_k(R_0)}Q.$$
We claim that
$$\Big\{x\in R_0: \sum_{Q\in \Sp(R_0)}\chi_Q(x)>k\Big\}=\Omega_{k+1}.$$
Notice that the function 
$$f_{R_0}(x)=\sum_{Q\in \Sp(R_0)} \chi_Q(x)$$
is an integer valued function that counts the number of cubes in $\Sp(R_0)$ that contain $x$.  With this in mind it is easy to see that
$$\Big\{x\in R_0: \sum_{Q\in \Sp(R_0)}\chi_Q(x)>k\Big\}\supseteq \Omega_{k+1}.$$
To see the reverse inclusion, note that if 
$$x\in \Big\{x\in R_0: \sum_{Q\in \Sp(R_0)}\chi_Q(x)>k\Big\},$$
then $x$ belongs to at least $k+1$ cubes of $\Sp(R_0)$, so $x$ must
belong to a cube in $\Ps_{k+1}(R_0)$.  Finally, by the sparsity condition on
the family $\Sp(R_0)$ and the disjointness of the families
$\Ps_k(R_0)$ we have
$$|\Omega_{k+1}|\leq \frac12|\Omega_k|\leq \frac{1}{4}|\Omega_{k-1}|\leq\cdots \leq \frac{1}{2^{k}}|\Omega_1|\leq \frac{1}{2^{k}}|R_0|.$$

%Suppose to the contrary that for some $k$ equality holds in
%\eqref{eqn:ss1}.  Since these sets are nested, we therefore have that
%
%\[ \Big|\Big\{x\in R_0: \sum_{Q\in \Sp(R_0)}\chi_Q(x)>1\Big\}\Big|
%=|R_0|, \]
%
%and so (up to a set of measure zero) every point $x\in R_0$ is contained
%in at least two different cubes in $\Sp(R_0)$.  Fix $x_0\in R_0$ and
%let $Q_0\in \Sp(R_0)$ be the largest such cube containing it.  Then for every $x\in
%Q_0$, there exists another cube $Q'\in \Sp(R_0)$ such that $x\in Q'$.
%By the maximality of $Q_0$, we must have that $Q'\subset Q_0$.  Since
%this is true for every $x\in Q_0$, we have that 
%
%\[ \Big|\bigcup_{\stackrel{Q'\in \Sp(R_0)}{Q'\subsetneq Q}} Q'\Big|=
%|Q|, \]
%
%and this contradicts the fact that $\Sp(R_0)$ is a sparse family.
\end{proof}

\begin{remark} \label{remk:disjoint} We note one identity from the
  proof of Lemma \ref{lemma:expdecyLeb} that we will use below:
$$\{x\in {R_0}:\sum_{\stackrel{Q\in \Sp}{Q\subseteq R_0}}\chi_Q(x)>k\}=\bigcup_{Q\in \Ps_{k+1}(R_0)}Q,$$
where $\Ps_{k+1}(R_0)$ is the collection of maximal cubes in $\Sp$ contained in $R_0$ at the $(k+1)$-th level down.  \end{remark}

\begin{proof}[Proof of Theorem \ref{thm:maintest}]
   
To prove~\eqref{eqn:mainest2sup},  fix $R\in \D$ and let $\Sp(R)=\{Q\in \Sp:Q\subseteq R\}$.  It  will
suffice to show that
\begin{equation}\label{eqn:keyest2}
\Big(\int_R I_\al^{\Sp(R)}(\chi_Ru)^{p'}\sigma\,dx\Big)^{1/p'}
\lesssim [u,\sigma]_{A_{s(p)}}^{1/q}[u]_{A_\infty'}^{1/p'}u(R)^{1/q'}.
\end{equation}

To estimate the operator $I_\al^{\Sp(R)}$ we need to decompose the
family $\Sp(R)$ into a collection of smaller sets.   The first step
allows us to ``freeze'' (i.e., gain local control of) the $A_{s(p)}$
constant of $u$.  For each $a\in \Z$ define
$$\Q^a:=\Big\{Q\in \Sp(R): 2^{a}<
\Big(\,\dashint_Q u\,dx\Big)^{\frac{1}{q}}\Big(\,\dashint_Q \sigma\,dx\Big)^{\frac{1}{p'}}\leq 2^{a+1}\Big\}.$$
The set $\Q^a$ is empty if $2^a>[u,\sigma]_{A_{s(p)}}^{1/q}$,
so we may assume that $$-\infty< a \leq \log_2 [u,\sigma]_{A_{s(p)}}^{1/q}=\Gamma(u).$$  In
particular, we have that
\begin{equation} \label{eqn:sum-a}
 \sum_{a=-\infty}^{\Gamma(u)} 2^a \lesssim [u,\sigma]_{A_{s(p)}}^{1/q}.
\end{equation}

Our next step is to perform a Corona decomposition of $\Sp(R)$ similar
to that in~\cite{MR2657437}.  Given $a$, let $C^a_0$ be the set of
maximal cubes in $\Q^a$.  For each $k\geq 1$, define the set $C^a_k$
by induction to be the (possibly empty) collection of cubes $Q\in
\Q^a$ such that following three criteria are satisfied:
\begin{enumerate}
\item there exists $P\in C_{k-1}^a$ containing $Q$,
\item the inequality
\begin{equation}\label{stopineq}|Q|^{\frac{\al}{n}}\dashint_Q u\,dx>2|P|^{\frac{\al}{n}}\dashint_P u\,dx\end{equation}
holds, 
\item and $Q$ is maximal with respect to inclusion in $\Q^a$.
\end{enumerate}
Set $\Ca^a=\bigcup_k C^a_k$; we refer this set as the collection of
stopping cubes for the Corona decomposition of $\Q^a$.  $\Ca^a$ can be
thought of as the collection cubes in $\Q^a$ whose fractional average
increases by a factor of two when passing from parent to child in~$\Ca^a$.  

By the maximality of the stopping cubes, given any $Q\in \Q^a$ there
exists a smallest $P\in \Ca^a$ such that $P\supseteq Q$ and the
reverse of inequality \eqref{stopineq},
\begin{equation}\label{revineq}
|Q|^{\frac{\al}{n}}\dashint_Q u\,dx\leq 2|P|^{\frac{\al}{n}}\dashint_P
u\,dx,
\end{equation}
holds.   Denote this cube $P$ by $\Pi^a(Q)$.  For each $P\in \Ca^a$
let
$$\Q^a(P)=\{Q\in \Q^a: \Pi^a(Q)=P\}.$$
Then inequality \eqref{revineq} holds for all $Q\in \Q^a(P)$.

Finally, we want to control  one more value: for every integer $b\geq 0$ and
$P\in \Ca^a$,  let $\Q^a_b(P)$ be the set of  $Q\in \Q^a(P)$ such that
\begin{equation}\label{freezefracavg}
2^{-b}|P|^{\frac{\al}{n}}\dashint_P u\,dx < 
|Q|^{\frac{\al}{n}}\dashint_Q u\,dx\leq
2^{-b+1}|P|^{\frac{\al}{n}}\dashint_P u\,dx.
\end{equation}

By the above definitions, we have that 

$$\Sp(R)=\bigcup_{a=-\infty}^{\Gamma(u)}\Q^a, \ \ 
\Q^a=\bigcup_{P\in \Ca^a}\Q^a(P), \ \ 
\Q^a(P)=\bigcup_{b=0}^\infty \Q^a_b(P),$$
and each of these unions is disjoint.  Therefore, we can decompose the
operator as follows:

\begin{multline*}
I_\al^{\Sp(R)}u=\sum_{a=-\infty}^{\Gamma(u)}\sum_{P\in \Ca^a}\sum_{b=0}^\infty \sum_{Q\in \Q^a_b(P)} |Q|^{\frac{\al}{n}}\dashint_Qu\,dx\cdot\chi_Q\\
\leq 2\sum_{a=-\infty}^{\Gamma(u)}\sum_{b=0}^\infty 2^{-b}\sum_{P\in \Ca^a} |P|^{\frac{\al}{n}}\dashint_P u\,dx \sum_{Q\in \Q^a_b(P)}\chi_Q.
\end{multline*}
For each $k\geq 0$, define 
$$E^a_b(k,P)=\Big\{x\in P : k<\sum_{Q\in \Q^a_b(P)}\chi_Q(x)\leq k+1\Big\}$$
and 
$$F^a_b(k,P)=\Big\{x\in P : \sum_{Q\in \Q^a_b(P)}\chi_Q(x)> k\Big\}.$$
Then we have that
\begin{multline*} 
I_\al^{\Sp(R)}u(x)
\lesssim\sum_{a=-\infty}^{\Gamma(u)}\sum_{b=0}^\infty 2^{-b}\sum_{k=0}^\infty 
(k+1)\sum_{P\in \Ca^a} |P|^{\frac{\al}{n}}\dashint_P u\,dx 
\cdot \chi_{E^a_b(k,P)}(x)\\
\leq\sum_{a=-\infty}^{\Gamma(u)}\sum_{b=0}^\infty 2^{-b}\sum_{k=0}^\infty
 (k+1)\sum_{P\in \Ca^a} |P|^{\frac{\al}{n}}\dashint_P u\,dx 
\cdot \chi_{F^a_b(k,P)}(x).
\end{multline*}
Hence, by Minkowski's inequality, 
\begin{multline}\label{reducedest}
\lefteqn{\left(\int_R(I_\al^{\Sp(R)}u)^{p'}\sigma\,dx\right)^{1/p'}}\\ 
\lesssim \sum_{a=-\infty}^{\Gamma(u)}\sum_{b=0}^\infty 2^{-b}
\sum_{k=0}^\infty(k+1)
\left(\int_R\Big(\sum_{P\in \Ca^a} 
|P|^{\frac{\al}{n}}\dashint_Pu\,dx
\cdot \chi_{F^a_b(k,P)}\Big)^{p'}\sigma\,dx\right)^{1/p'}. 
\end{multline}

\smallskip

To estimate the last term, we will first show that for each
$a,\,b$ and $k$, 
\begin{multline} \label{pinsidesum}
\left(\int_R\Big(\sum_{P\in \Ca^a} |P|^{\frac{\al}{n}}
\dashint_Pu\,dx\cdot \chi_{F^a_b(k,P)}\Big)^{p'}\sigma\,dx\right)^{1/p'}\\
 \lesssim\left(\sum_{P\in \Ca^a}
   \Big(|P|^{\frac{\al}{n}}\dashint_Pu\,dx\Big)^{p'}
\cdot \sigma(F^a_b(k,P))\right)^{1/p'}.
\end{multline}
To prove this, note that since the cubes in $\Ca^a$ are stopping
cubes,  the set of $x\in R$ that belongs to infinitely many $P\in \Ca^a$ has
measure zero. Fix  $x\in R$ not in this set, and let $\{P_i\}_{i=0}^m$ be the
stopping cubes such that $ P_0\subset P_1\subset \cdots \subset
P_m \subset R$ and $x\in F_b^a(k,P_i)$.   By the 
definition of the stopping cubes we have that
$$|P_i|^{\frac{\al}{n}}\dashint_{P_i} u\,dx< 2^{-i}|P_0|^{\frac{\al}{n}}\dashint_{P_0}u\,dx.$$
Therefore, 
\begin{multline*}
\Big(\sum_{P\in \Ca^a} |P|^{\frac{\al}{n}}\dashint_Pu\,dx
\cdot \chi_{F^a_b(k,P)}(x)\Big)^{p'}
=\Big(\sum_{i=0}^m |P_i|^{\frac{\al}{n}}\dashint_{P_i}u\,dx\Big)^{p'}\\
<\Big(\sum_{i=0}^m
2^{-i}\Big)^{p'}\Big(|P_0|^{\frac{\al}n}\dashint_{P_0} u\,dx\Big)^{p'}
< 2^{p'}\sum_{P\in \Ca^a}\Big(|P|^{\frac{\al}{n}}\dashint_P
u\,dx\Big)^{p'} \chi_{F_a^b(k,P)}(x).
\end{multline*}
If we integrate this quantity over $R$ with respect to $\sigma\,dx$ we
get inequality~\eqref{pinsidesum}.

\smallskip

To continue, suppose for a moment that we have the exponential decay
estimate
\begin{equation}\label{eqn:expdecay}
\sigma(F_b^a(k,P))\lesssim 2^{-ck} \sigma(P).
\end{equation}
Then by inequalities~\eqref{pinsidesum} and~\eqref{eqn:expdecay} we
have that
\begin{align}
&
\left(\int_R(I_\al^{\Sp(R)}u)^{p'}\sigma\,dx\right)^{1/p'} \notag \\ 
& \qquad \qquad \lesssim \sum_{a=-\infty}^{\Gamma(u)}\sum_{b=0}^\infty 2^{-b}
\sum_{k=0}^\infty2^{-ck}(k+1)
\left(\sum_{P\in \Ca^a}
  \Big(|P|^{\frac{\al}{n}}\dashint_Pu\,dx\Big)^{p'}
\cdot \sigma(P)\right)^{1/p'} \notag \\
& \qquad \qquad \lesssim \sum_{a=-\infty}^{\Gamma(u)}\left(\sum_{P\in \Ca^a}
  \Big(|P|^{\frac{\al}{n}}\dashint_Pu\,dx\Big)^{p'}
\cdot \sigma(P)\right)^{1/p'}.  \label{eqn:almost-there}
\end{align}

To estimate the final sum, note first that by the definition of
$\Q^a$, if $P \in \Ca^a$, 
\[ \left(\avgint_P u\,dx\right)^{s(q')-1} 
\left( \avgint_P \sigma \,dx \right) \lesssim 2^{ap'}. \]
Therefore, we have that 
\begin{multline}
 \Big(|P|^{\frac{\al}{n}}\dashint_Pu\,dx\Big)^{p'}\cdot \sigma(P) 
=\Big(\frac{1}{u(P)^{1-\frac{\al}{n}}}\int_P \chi_R \cdot
u\,dx\Big)^{p'} 
\frac{u(P)^{s(q')}\sigma(P)}{|P|^{s(q')}} \\
\lesssim 2^{ap'}\Big(\frac{1}{u(P)^{1-\frac{\al}{n}}}
\int_P \chi_R \cdot u\,dx\Big)^{p'} u(P).\label{eqn:2aest}
\end{multline}

For cubes $Q\in \D$, define the sequence $\{c_Q\}$ by 
$$c_Q=\left\{\begin{array}{cc}u(Q) & Q\in \Ca^a \vspace{3mm} \\  0 & Q\notin \Ca^a. \end{array}\right.$$
We claim this is a Carleson sequence and 
\begin{equation}\label{eqn:carleson}
\sum_{Q\subseteq P} c_{Q}\lesssim [u]_{A_\infty'} u(P).
\end{equation}
Fix a cube $P$; since $\Ca^a\subset \Sp(R)$, 
\[ 
\sum_{Q\subseteq P} c_Q
\leq \sum_{\substack{Q\in \Sp(R)\\ Q\subseteq P}} u(Q)
\lesssim \sum_{\substack{Q\in \Sp(R)\\ Q\subseteq P}} \frac{u(Q)}{|Q|}|E(Q)|
\leq \int_P M(\chi_P u)\,dx\leq [u]_{A_\infty'}u(P).
\]

Therefore, if we combine inequalities~\eqref{eqn:almost-there}
and~\eqref{eqn:2aest}, then by Lemmas~\ref{weightedmax}
and~\ref{Carleson} and inequality~\eqref{eqn:sum-a} we have that 
\begin{align*}
 \left(\int_R(I_\al^{\Sp(R)}u)^{p'}\sigma\,dx\right)^{1/p'}
& \lesssim \sum_{a=-\infty}^{\Gamma(u)}2^{a}\left( \sum_{P\in \Ca^a} 
\Big(\frac{1}{u(P)^{1-\frac{\al}{n}}}
\int_P \chi_R \cdot u\,dx\Big)^{p'} u(P)\right)^{1/p'} \\
&\lesssim [u]_{A_\infty'}^{1/p'}\bigg(\sum_{a=-\infty}^{\Gamma(u)}2^{a}\bigg)
\left(\int_{\R^n} M_{\al,u}^\D(\chi_Ru)^{p'}u\,dx\right)^{1/p'}\\
&\lesssim [u,\sigma]_{A_{s(p)}}^{1/q}[u]_{A_\infty'}^{1/p'}u(R)^{1/q'}.
\end{align*}

\medskip

To complete the proof it remains to prove inequality
\eqref{eqn:expdecay}:   for $a$, $b$, $k$, fixed and $P\in \Ca^a$, 
$$\sigma(F^a_b(k,P)) = \sigma\Big(\Big\{x\in P:
\sum_{Q\in\Q^a_b(P)}\chi_Q(x)>k\Big\}\Big)\lesssim 2^{-ck}\sigma(P).$$

If $x\in F^a_b(k,P)$, then clearly $x\in Q$ for some $Q\in \Q^a_b(P)$.
Therefore, if we  let $\MM$ be the
collection of maximal, disjoint cubes $Q\in \Q^a_b(P)$ contained in $P$, we have that 
\begin{multline}
\sigma\Big(\Big\{x\in P: \sum_{Q\in\Q^a_b(P)}\chi_Q(x)>k\Big\}\Big)
=\sum_{M\in \MM}\sigma\Big(\Big\{x\in M:
\sum_{\substack{Q\in\Q^a_b(P)\\ Q\subseteq M}}\chi_Q(x)>k\Big\}\Big). \label{Msum} \end{multline}
Fix $M\in \MM$ and notice that the family of cubes $Q\in \Q^a_b(P)$ is a sparse family of cubes contained in $P$. For each $M\in\MM$, as in Lemma \ref{lemma:expdecyLeb} (see Remark \ref{remk:disjoint}) we may write
$$\{x\in M: \sum_{\substack{Q\in\Q^a_b(P)\\ Q\subseteq M}}\chi_Q(x)>k\}=\bigcup_{L\in \Ps_{k+1}(M)} L$$
where the union is made up of maximal cubes in contained in $M$ at the $(k+1)$-th level down.
For any cube, $Q\in \Q^a_b(P)$ (and in particular if $Q=L\in \Ps_{k+1}(M)$ or $M\in \MM$) by
the definition of $\Q^a(P)$ and \eqref{freezefracavg} we have that
\begin{equation}\label{eqn:lebweightequiv}
\sigma(Q)\simeq 2^{ap'}2^{b\frac{p'}{q}}
\Big(\frac{|P|^{1-\frac{\al}{n}}}{u(P)}\Big)^{\frac{p'}{q}}|Q|^{\frac{p'}{q}\frac{\al}{n}+1}.
\end{equation}
Therefore, we can estimate as follows:  by  one side of inequality
\eqref{eqn:lebweightequiv}, with $Q=L$ 
\begin{align*}
\lefteqn{\sigma(\{x\in M: \sum_{\substack{Q\in\Q^a_b(P)\\ Q\subseteq M}}\chi_Q(x)>k\})
 =\sum_{L\in \Ps_{k+1}(M)}\sigma(L)}\\
& \lesssim
2^{ap'}2^{b\frac{p'}{q}}\Big(\frac{|P|^{1-\frac{\al}{n}}}{u(P)}\Big)^{\frac{p'}{q}}
\sum_{L\in \Ps_{k+1}(M)} |L|^{\frac{p'}{q}\frac{\al}{n}+1}; \\
\intertext{since $1+\frac{\al}{n}\frac{p'}{q}\geq 1$,}
& \leq
2^{ap'}2^{b\frac{p'}{q}}\Big(\frac{|P|^{1-\frac{\al}{n}}}{u(P)}\Big)^{\frac{p'}{q}}
\Big(\sum_{L\in \Ps_{k+1}(M)} |L|\Big)^{\frac{p'}{q}\frac{\al}{n}+1} \\
& \leq
2^{ap'}2^{b\frac{p'}{q}}\Big(\frac{|P|^{1-\frac{\al}{n}}}{u(P)}\Big)^{\frac{p'}{q}}
|\{x\in M:\sum_{\substack{Q\in\Q^a_b(P)\\Q\subseteq M}}\chi_Q(x)>k\}|^{\frac{p'}{q}\frac{\al}{n}+1};
\intertext{by inequality \eqref{eqn:expdecayLeb} and the other half of
 inequality \eqref{eqn:lebweightequiv} with $Q=M$,}
&\leq 2^{ap'}2^{b\frac{p'}{q}}\Big(\frac{|P|^{1-\frac{\al}{n}}}{u(P)}\Big)^{\frac{p'}{q}} 2^{-k(\frac{p'}{q}\frac{\al}{n}+1))}|M|^{\frac{p'}{q}\frac{\al}{n}+1}\\
&\lesssim 2^{-k(\frac{p'}{q'}\frac{\al}{n}+1)} \sigma(M).
\end{align*}
If we combine this inequality with \eqref{Msum}, we get
\begin{align*}
\sigma\Big(\Big\{x\in P: \sum_{Q\in\Q^a_b(P)}\chi_Q(x)>k\Big\}\Big)&\leq \sum_{M\in \MM}\sigma(\{x\in M: \sum_{\substack{Q\in\Q^a_b(P)\\ Q\subseteq M}}\chi_Q(x)>k\})\\
&\lesssim 2^{-ck}\sum_{M\in \MM} \sigma(M)\\
&\leq 2^{-ck}\sigma(P)
\end{align*}
as desired.  
\end{proof}

\section{Logarithmic bump condtions} \label{proofs}

In this section we prove Theorems~\ref{thm:mainweak},
\ref{thm:weakloglog}, \ref{thm:mainstrong}
and~\ref{thm:strongloglog}.   We first consider the results for log
bumps.

\begin{theorem} \label{thm:twoweighttesting}
Fix $\alpha$,  $0<\al<n$, and $1<p< q < \infty$ such that
$\frac{p'}{q'}(1-\frac{\al}{n})\geq 1$.  Suppose 
$\Phi(t)=t^{q}\log(e+t)^{q-1+\delta}$ for some $\delta>0$ and
$(u,\sigma)$ is a pair of weights that satisfies
$$K=\sup_{Q}|Q|^{\frac{\al}{n}+\frac{1}{q}-\frac1p}\|u^{\frac{1}{q}}\|_{\Phi,Q}
\|\sigma^{\frac{1}{p'}}\|_{p',Q}<\infty.$$
Then for every dyadic grid $\D$ with sparse subset $\Sp$,
\begin{equation} \label{eqn:twt1}
[\sigma,u]_{(I^\Sp_\al)^{q',p'}}^\D \lesssim K.
\end{equation}
Similarly, if $\frac{q}{p}\Big(1-\frac{\al}{n}\Big)\geq 1$,
$\Psi(t)=t^{p'}\log(e+t)^{p'-1+\delta}$ and the pair $(u,\sigma)$ satisfies
$$K=\sup_{Q}|Q|^{\frac{\al}{n}+\frac{1}{q}-\frac1p}\|u^{\frac{1}{q}}\|_{p,Q}
\|\sigma^{\frac{1}{p'}}\|_{\Psi,Q}<\infty,$$
then for every dyadic grid $\D$ with sparse subset $\Sp$,
\begin{equation} \label{eqn:twt2}
[u,\sigma]_{(I^\Sp_\al)^{p,q}}^\D \lesssim K.
\end{equation}
\end{theorem}

As in the previous section, Theorems~\ref{thm:mainweak}
and~\ref{thm:mainstrong} follow immediately from
Theorem~\ref{thm:twoweighttesting} and the results in
Section~\ref{preliminaries}. 

\smallskip

For the proof of Theorem~\ref{thm:twoweighttesting} we need three
lemmas.  The first is classical:  see~\cite{MR2797562,rao-ren91}.  

\begin{lemma} \label{lemma:holder}
Given a Young function $\Phi$, for every cube $Q$ and functions
$f$ and $g$, 
$$\dashint_Q |fg|\,dx\lesssim \|f\|_{\Phi,Q}\|g\|_{\bar{\Phi},Q}.$$
\end{lemma}

To state the second, we need a definition.   Given a Young function $\Phi$  define the corresponding maximal function,
$$M_\Phi f(x)=\sup_{Q\ni x} \|f\|_{\Phi,Q}.$$
The
following result is due to Perez \cite{MR1327936} (also
see~\cite{MR2797562}).  

\begin{lemma} \label{lemma:perez-max}
Given a Young function $\Phi$ and any $p$, $1<p<\infty$,
$$\|M_\Phi f\|_{L^p(\R^n)}\lesssim \|f\|_{L^p(\R^n)}$$
if and only if $\Phi\in B_p$.
\end{lemma}

The third lemma is from \cite{CRV2012}.

\begin{lemma} \label{lemma:crv} 
Given $q$, $1<q<\infty$, let
  $\Phi(t)=t^q\log(e+t)^{q-1+\delta}$ and
  $\Phi(t)=t^q\log(e+t)^{q-1+\delta/2}$.  Then there exists 
  $\gamma$, $0<\gamma<1$,  such that for every cube $Q$,
\begin{equation}\label{wiggleroom}
\|u^{\frac{1}{q}}\|_{\Phi_0,Q}\lesssim 
\|u^{\frac{1}{q}}\|_{{\Phi},Q}^{1-\gamma}\|u^{\frac1q}\|_{q,Q}^{\gamma}.
\end{equation}
\end{lemma}

\begin{proof}[Proof of Theorem \ref{thm:twoweighttesting}] 
The proof is very similar to the proof of Theorem~\ref{thm:maintest}
and we will sketch briefly those parts that are the same.  As before,
we will only prove~\eqref{eqn:twt1}; the proof of \eqref{eqn:twt2} is
the same after making the obvious changes.
Fix $R\in \D$.  Then it will suffice to prove that 
\begin{equation} \label{keyest}
\left(\int_R (I_\al^{\Sp(R)}u)^{p'}\sigma\,dx\right)^{1/p'}
\lesssim K u(R)^{1/q'}.
\end{equation}

We decompose the family $\Sp(R)$; however, in the first step there is
a significant difference.  For $a\in \Z$ define
$$\Q^a:=\Big\{Q\in \Sp(R): 2^{a}<|Q|^{\frac{\al}{n}+\frac1q-\frac1p}
\Big(\,\dashint_Q u\,dx\Big)^{\frac{1}{q}}
\Big(\,\dashint_Q \sigma\,dx\Big)^{\frac{1}{p'}}\leq 2^{a+1}\Big\}.$$

Since $\|u^{\frac{1}{q}}\|_{q,Q}\leq \|u^{\frac{1}{q}}\|_{\Phi,Q}$,
by our assumption on $(u,\sigma)$, the set $\Q^a$ is empty if
$a>\log_2 K$.   Therefore, we will sum over $a$ contained in the set
$$\Omega(K)=\Z\cap(-\infty,\log_2K].$$

With this definition of $\Q^a$, for $a\in \Omega(K)$, define $\Ca^a$,
$\Q^a(P)$ and $\Q^a_b(P)$ exactly as before.  Then the same argument
shows that
\begin{multline} \label{eqn:first-step}
\lefteqn{\left(\int_R(I_\al^{\Sp(R)}u)^{p'}
\sigma\,dx\right)^{1/p'}}\\ 
\lesssim \sum_{a\in \Omega(K)}\sum_{b=0}^\infty 2^{-b}
\sum_{k=0}^\infty(k+1)\left(\sum_{P\in \Ca^a} 
\Big(|P|^{\frac{\al}{n}}\dashint_Pu\,dx\Big)^{p'}\cdot \sigma(F^a_b(k,P))\right)^{1/p'}.
\end{multline}
As before we have that 
\begin{equation*}
\sigma(F_b^a(k,P))\lesssim 2^{-ck} \sigma(P);
\end{equation*}
the proof is essentially the same as the proof of \eqref{eqn:expdecay}: the key difference is
that for $Q\in \Q^a_b(P)$ we now have
$$\sigma(Q)\simeq 2^{ap'}2^{b\frac{p'}{q}}\Big(\frac{|P|^{1-\frac{\al}{n}}}{u(P)}\Big)^{\frac{p'}{q}}|Q|^{\frac{p'}{q'}(1-\frac{\al}{n})}.$$
Moreover, we note that it is in this part of the proof that we use the
assumption that $\frac{p'}{q'}(1-\frac{\alpha}{n})\geq 1$ in order to pull
this power out of the sum.  (Cf.~ the argument immediately following~\eqref{eqn:lebweightequiv}.)
If we substitute this into \eqref{eqn:first-step}, we can now sum in
$b$ and $k$ to get
\begin{multline} \label{finalsum2}
\left(\int_R(I_\al^{\Sp(R)}u)^{p'}\sigma\,dx\right)^{1/p'} \\
\lesssim \sum_{a\in \Omega(K)} \sum_{b=0}^\infty \sum_{k=0}^\infty
2^{-ck}(k+1) 
\left(\sum_{P\in \Ca^a}
  \Big(|P|^{\frac{\al}{n}}\dashint_Pu\,dx\Big)^{p'}
\cdot \sigma(P)\right)^{1/p'} \\
\lesssim \sum_{a\in \Omega(K)}
\left(\sum_{P\in \Ca^a}
  \Big(|P|^{\frac{\al}{n}}\dashint_Pu\,dx\Big)^{p'}
\cdot \sigma(P)\right)^{1/p'}.
\end{multline}

To evaluate the inner sum we apply Lemma~\ref{lemma:crv}:
since $\sigma(P) = \|\sigma^{\frac{1}{p'}}\|_{p',P}^{p'}|P|$,
\begin{align*}
& \Big(|P|^{\frac{\al}{n}}\dashint_Pu\,dx\Big)^{p'}\sigma(P) \\
& \qquad =|P|^{\frac{\al}{n}p'}\left(\dashint_Pu\,dx\right)^{p'}\sigma(P)\\
& \qquad \lesssim |P|^{\frac{\al}{n}p'}\|u^{\frac{1}{q}}\|_{\Phi_0,P}^{p'}
\|u^{\frac{1}{q'}}\|_{\bar{\Phi}_0,P}^{p'}\sigma(P)\\
& \qquad \lesssim|P|^{\frac{\al}{n}p'}\|u^{\frac{1}{q}}\|_{\Phi,P}^{(1-\gamma) p'}
\|u^{\frac1q}\|_{q,P}^{p'\gamma}\|u^{\frac{1}{q'}}\|_{\bar{\Phi}_0,P}^{p'}\sigma(P)\\
& \qquad \lesssim|P|^{\frac{\al}{n}p'+\frac{p'}{q}-\frac{p'}{p}}\|u^{\frac{1}{q}}\|_{\Phi,P}^{ (1-\gamma)p'}
\|\sigma^{\frac{1}{p}}\|_{p,P}^{(1-\gamma)p'}
\|u^{\frac1q}\|_{q,P}^{p'\gamma}
\|\sigma^{\frac{1}{p}}\|_{p,P}^{\gamma p'}
\|u^{\frac{1}{q'}}\|_{\bar{\Phi}_0,P}^{p'}|P|^{1+\frac{p'}{p}-\frac{p'}{q}} \\
& \qquad \lesssim K^{(1-\gamma)p'}2^{ap'\gamma} \|u^{\frac{1}{q'}}\|_{\bar{\Phi}_0,P}^{p'}|P|^{\frac{p'}{q'}}.
\end{align*}
Therefore, the inner sum in \eqref{finalsum2} becomes
\begin{align*}
\left(\sum_{P\in \Ca^a} \Big(|P|^{\frac{\al}{n}}\dashint_P
  u\,dx\Big)^{p'} \sigma(P)\right)^{1/p'}
&\lesssim K^{1-\gamma}2^{\gamma a}\left(\sum_{P\in \Ca^a} 
\|u^{\frac1{q'}}\|_{\bar{\Phi}_0,P}^{p'}|P|^{\frac{p'}{q'}}\right)^{1/p'};\\
\intertext{since $q'/p'\leq 1$ we may pull this power into the sum, and since the cubes in $\Ca^a$ are a
  sparse family we can apply inequality~\eqref{eqn:equivmeasure} to get}
& \lesssim K^{1-\gamma}2^{\gamma a}\left(\sum_{P\in \Ca^a} 
\|u^{\frac1{q'}}\|_{\bar{\Phi}_0,P}^{q'}|E(P)|\right)^{1/q'} \\
& \leq K^{1-\gamma}2^{\gamma a}\left(\sum_{P\in \Ca^a} 
\int_{E(P)}M_{\bar{\Phi}_0}(u^{\frac{1}{q'}}\chi_R)^{q'}\,dx\right)^{1/q'} \\
&\leq K^{1-\gamma}2^{\gamma a}\left(\int_{\R^n} 
M_{\bar{\Phi}_0}(u^{\frac1{q'}}\chi_R)^{q'}\,dx\right)^{1/q'} \\
&\lesssim K^{1-\gamma}2^{\gamma a}u(R)^{1/q'}.
\end{align*}
In the final inequality we used Lemma~\ref{lemma:perez-max}; we can do
this because 
$$\bar{\Phi}_0(t)\approx t^{q'}\log(e+t)^{-1-\frac{\delta}{2(q-1)}}$$
satisfies the $B_{q'}$ condition.

Finally, given the factor $2^{\gamma a}$, if we plug this estimate
into \eqref{finalsum2}, the final sum in $a$  converges and is bounded
by $K^\gamma$. We therefore get the desired
inequality and  this completes the proof.
\end{proof}

\begin{remark}
In the proof of Theorem~\ref{thm:twoweighttesting} we actually get a
sharper, ``mixed'' estimate.  If we define
$$[u,\sigma]_{A_{p,q}^\al}:=\sup_Q |Q|^{\frac{\al}{n}+\frac1q-\frac1p}
\Big(\,\dashint_Q u\,dx\Big)^{\frac{1}{q}}
\Big(\,\dashint_Q \sigma\,dx\Big)^{\frac{1}{p'}}$$
then a careful analysis of the constants in the proof shows that we
actually get the sharper bound
\[ [\sigma,u]_{(I^\Sp_\al)^{q',p'}}^\D\lesssim K^{1-\gamma} [u,\sigma]^\gamma_{A^{\alpha}_{p,q}}\leq K. \]
Moreover, if we modify the definition of $\Phi_0$ by replacing
$\delta/2$ by a suitable constant, we can prove that for any
$\epsilon$, $0<\epsilon<1$,  we can get the bound
\[ [\sigma,u]_{(I^\Sp_\al)^{q',p'}}^\D\leq C(\epsilon)K^{1-\epsilon} [u,\sigma]^\epsilon_{A^{\alpha}_{p,q}}, \]
where $C(\epsilon)\rightarrow \infty$ as $\epsilon\rightarrow 1$.
Details are left to the interested reader. 
\end{remark}

\medskip

We now prove Theorems~\ref{thm:weakloglog}
and~\ref{thm:strongloglog}.   To do so we need to extend
Theorem~\ref{thm:twoweighttesting} to the scale of loglog bumps.  

\begin{theorem} \label{thm:loglog}
The conclusions of Theorem~\ref{thm:twoweighttesting} remain true with
the same hypotheses if we replace the Young functions $\Phi$ and
$\Psi$ with 
\begin{gather*}
\Phi(t) = t^q\log(e+t)^{q-1}\log\log(e^e+t)^{q-1+\delta}, \\
\Psi(t) = t^{p'}\log(e+t)^{p'-1}\log\log(e^e+t)^{p'-1+\delta},
\end{gather*}
where $\delta>0$ is taken sufficiently large.
\end{theorem}

We will briefly sketch the proof of Theorem \ref{thm:loglog}, as it is
very similar to the proof of Theorem \ref{thm:twoweighttesting}.  The
main difference is that we must replace Lemma~\ref{lemma:crv} with the
following result which was also proved in~\cite{CRV2012}.

\begin{lemma} \label{lem:loglog} 
Given $q$, $1<q<\infty$,  let 
$$\Phi(t)=t^q\log(e+t)^{q-1}\log\log(e^e+t)^{q-1+\delta}$$
and
$$\Phi_0(t)=t^q\log(e+t)^{q-1}\log\log(e^e+t)^{q-1+\delta/2}.$$
Then 
\[ 
\|u^{\frac1q}\|_{\Phi_0,Q}\lesssim \|u^{\frac1q}\|_{\Phi,Q}\cdot
\phi\Bigg(\frac{\|u^{\frac{1}{q}}\|_{q,Q}}{\|u^{\frac1q}\|_{\Phi,Q}}\Bigg) 
\]
where $\phi(t)=\log(C/t)^{-\kappa}$ with $\kappa, C$ constants.
Moreover if $\delta>0$ is sufficiently large, then $\kappa>1$.  
\end{lemma}

\begin{proof}[Sketch of the proof of Theorem \ref{thm:loglog}] 
We use the same notation as in Theorem \ref{thm:twoweighttesting}.
The proof is identical until estimate \eqref{finalsum2}:
\begin{multline*}
\lefteqn{\left(\int_R(I_\al^{\Sp(R)}u)^{p'}\sigma\,dx\right)^{1/p'}}\\
\lesssim \sum_{a\in \Omega(K)}\sum_{b=0}^\infty 2^{-b}
\sum_{k=0}^\infty2^{-ck}(k+1)\left(\sum_{P\in \Ca^a}
 \Big(|P|^{\frac{\al}{n}}\dashint_Pu\,dx\Big)^{p'}\cdot \sigma(P)\right)^{1/p'}.
\end{multline*}
Once again we can sum the series in $b$ and $k$, and the problem is to
sum the series in $a$, where
$$a\in \Omega(K)=\Z\cap(-\infty,\log_2 K]. $$

At this stage we use Lemma~\ref{lem:loglog} to estimate the last sum.
We have that
\begin{multline*}
  \Big(|P|^{\frac{\al}{n}}\dashint_Pu\,dx\Big)^{p'}
\sigma(P)=|P|^{\frac{\al}{n}p'}\left(\dashint_Pu\,dx\right)^{p'}\sigma(P)
 \lesssim |P|^{\frac{\al}{n}p'}\|u^{\frac{1}{q}}\|_{\Phi_0,P}^{p'}
\|u^{\frac{1}{q'}}\|_{\bar{\Phi}_0,P}^{p'}\sigma(P)\\
\lesssim\left(|P|^{\frac{\al}{n}+\frac{1}{q}-\frac{1}{p}}
\|u^{\frac1q}\|_{\Phi,P}\Big(\dashint_P
  \sigma\,dx\Big)^{\frac{1}{p'}}\right)^{p'}\cdot
\phi\bigg(\frac{\|u^{\frac{1}{q}}\|_{q,P}}{\|u^{\frac1q}\|_{\Phi,P}}\bigg)^{p'}\cdot
\|u^{\frac{1}{q'}}\|_{\bar{\Phi}_0,P} |P|^{\frac{p'}{q'}}.
\end{multline*}
To estimate this term we need to further divide the sum over $P\in
\Ca^a$.  Recall that if $P\in \Ca^a$, then
$$2^{a-1}<|P|^{\frac{\al}{n}+\frac{1}{q}-\frac{1}{p}}\Big(\,\dashint_P
u\,dx\Big)^{\frac{1}{q}}
\Big(\,\dashint_P \sigma\,dx\Big)^{\frac{1}{p'}}\leq 2^a.$$
Moreover, if $P\in \Ca^a$ then since $\|u^{\frac{1}{q}}\|_{q,P}
\leq \|u^{\frac{1}{q}}\|_{\Phi,P}$, there exists an integer $c$,
$a\leq c\leq \log_2K$, such that
\begin{equation}\label{cfreeze}
2^{c-1}<|P|^{\frac{\al}{n}+\frac{1}{q}-\frac{1}{p}}\|u^{\frac1q}\|_{\Phi,P}
\Big(\dashint_P \sigma\,dx\Big)^{\frac{1}{p'}}\leq 2^c.
\end{equation}
For each such $a$ and $c$ let $\Ca^a_c$ be the collection of all cubes
$P \in \Ca^a$ such that \eqref{cfreeze} holds.  Then we can estimate
as follows:
\begin{align*}
\Bigg(\sum_{P\in \Ca^a}
\Big(|P|^{\frac{\al}{n}}\dashint_Pu\,dx\Big)^{p'}\cdot
\sigma(P)\Bigg)^{\frac{1}{p'}}
 &=\Bigg(\sum_{\substack{c\in \Omega(K)\\ c\geq a}}\sum_{P\in \Ca^a_c}
 \Big(|P|^{\frac{\al}{n}}\dashint_Pu\,dx\Big)^{p'}\cdot \sigma(P)\Bigg)^{\frac{1}{p'}}\\
&\lesssim \sum_{\substack{c\in \Omega(K)\\ c\geq a}}\Bigg(\sum_{P\in
  \Ca^a_c}
\left(|P|^{\frac{\al}{n}+\frac{1}{q}-\frac{1}{p}}\|u^{\frac1q}\|_{\Phi,P}
\Big(\dashint_P \sigma\,dx\Big)^{\frac{1}{p'}}\right)^{p'} \\
& \qquad \cdot \phi\left(\frac{\|u^{\frac1q}\|_{q,P}}{\|u^{\frac1q}\|_{\Phi,P}}\right)^{p'}
\times \|u^{\frac{1}{q'}}\|_{\bar{\Phi}_0,P} |P|^{\frac{p'}{q'}} \Bigg)^{\frac{1}{p'}}\\
&\lesssim\sum_{\substack{c\in \Omega(K)\\c\geq a}}2^c
\Bigg(\sum_{P\in \Ca^a_c} \phi\Bigg(\frac{\|u^{\frac1q}\|_{q,P}}{\|u^{\frac1q}\|_{\Phi,P}}\Bigg)^{p'} 
 \cdot \|u^{\frac{1}{q'}}\|_{\bar{\Phi}_0,P}^{p'}|P|^{\frac{p'}{q'}}\Bigg)^{\frac{1}{p'}}.
\end{align*}

If $P\in \Ca^a_c$, then 
\[
\frac{\|u^{\frac1q}\|_{q,P}}{\|u^{\frac1q}\|_{\Phi,P}} \simeq 2^{a-c},
\]
and so
$$\phi\Bigg(\frac{\|u^{\frac1q}\|_{q,P}}{\|u^{\frac1q}\|_{\Phi,P}}\Bigg) 
\simeq \frac{1}{(1+c-a)^\kappa }.$$
Given this  the rest of proof proceeds exactly as before:  we have
that
$$\Bigg(\sum_{P\in \Ca^a_c} \|u^{\frac{1}{q'}}\|_{\bar{\Phi}_0,P}^{p'}
|P|^{\frac{p'}{q'}}\Bigg)^{\frac{1}{p'}}
\lesssim \bigg( \int_{\R^n}
M_{\bar{\Phi}_0}(u^{\frac{1}{q'}}\chi_R)^{q'}\,dx\bigg)^{1/q'}\lesssim u(R)^{1/q'},$$
since
$$\bar{\Phi}_0(t)\simeq \frac{t^{q'}}{\log(e+t)\log\log(e^e+t)^{1+\frac{\delta}{2(q-1)}}}$$
belongs to $B_{q'}$.  Moreover, the double sum
$$\sum_{a\in \Omega(K)}\sum_{\substack{c\in \Omega(K)\\ c\geq
      a}}\frac{2^c}{(1+c-a)^\kappa}
= \sum_{c=-\infty}^{\log_2 K} 2^c\sum_{a=-\infty}^c
\frac{1}{(1+c-a)^\kappa} \lesssim K
$$
converges if we assume that $\delta$ is large enough that  $\kappa>1$.
\end{proof}

\bibliographystyle{plain}
\bibliography{CruzUribeMoen}

\end{document}